\documentclass[12pt,a4paper]{article}

\usepackage[english]{babel} 
\usepackage{ucs}
\usepackage[utf8]{inputenc}

\RequirePackage{color} 

\usepackage[pdftex,colorlinks=true,urlcolor=blue,pdfstartview=FitH]{hyperref}
\usepackage[pdftex]{graphicx}
\DeclareGraphicsExtensions{.pdf, .png, .jpg, .pdftex} 
\RequirePackage{amsthm,amsmath,amssymb,amsfonts,latexsym,mathrsfs}

\newcommand{\ee}{\mathbf{e}}

\newcommand{\bt}{b_{T}}

\newcommand{\RR}{\mathcal{R}}
\newcommand{\ZZ}{\mathcal{Z}}

\newcommand{\NN}{\mathcal{N}}

\newcommand{\T}{\mathbb{T}}
\newcommand{\R}{\mathbb{R}}
 
\newcommand{\M}{\mathbb{M}}
\newcommand{\E}{\mathbb{E}}

\newcommand{\TT}{\mathcal{T}}
\def\build#1_#2^#3{\mathrel{ \mathop{\kern 0pt#1}\limits_{#2}^{#3}}}

\def\d{\mathrm{d}}

\def\eps{\varepsilon}

\renewcommand{\P}{\mathbb{P}}

\newtheorem{thm}{Theorem}
\newtheorem{lmm}[thm]{Lemma}
\newtheorem{prp}[thm]{Proposition}

\usepackage{fancyhdr}

\usepackage[outerbars]{changebar} 

\title{Long Brownian bridges in hyperbolic spaces converge to
  Brownian trees}

\author{Xinxin Chen\thanks{Institut Camille Jordan, Universit\'e Claude
    Bernard Lyon 1}\ \& Grégory Miermont\thanks{Ecole Normale
    Supérieure de Lyon et Institut Universitaire de France. Research
    supported in part by ANR grants GrAAl (ANR-14-CE25-0014) and
    Liouville (ANR-15-CE40-0013)}\\
Université de Lyon}

\begin{document}

\selectlanguage{english}

\maketitle

\begin{abstract}
  We show that the range of a long Brownian bridge in the hyperbolic
  space converges after suitable renormalisation to the Brownian
  continuum random tree. This result is a relatively
  elementary consequence of
  \begin{itemize}
  \item A theorem by Bougerol and Jeulin,
  stating that the rescaled radial process converges to the normalized
  Brownian excursion, 
\item A property of invariance under re-rooting,
  \item
The hyperbolicity of the ambient space in the sense of
  Gromov.
\end{itemize}
A similar result is obtained for the rescaled infinite Brownian loop
in hyperbolic space. 
\end{abstract}

\section{Introduction}\label{sec:introduction}

\subsection{Brownian bridges in hyperbolic space}\label{sec:brown-bridg-hyperb}

This work deals with geometric properties of the range of long
Brownian bridges in hyperbolic space.  For $d\geq 2$, let $H=H_d$ be
the $d$-dimensional hyperbolic space, and let $o$ be a distinguished
point taken as origin. For every $T>0$ we let $\bt$ be the Brownian
bridge on $H$ from the origin $o$ and with duration
$T$. Heuristically, this can be seen as Brownian motion $(B(t),t\geq
0)$ (the diffusion on $H$ whose generator is half the
Laplace-Beltrami operator) in restriction to the interval $[0,T]$ and
conditioned on the event $\{B(0)=B(T)=o\}$.

There are several natural (and equivalent) ways to make sense of this
singular conditioning. Let $p_t(x,y)$ be the transition densities for standard
Brownian motion on $H$, with respect to the standard volume measure
$m$ on $H$. Then the finite-dimensional distributions of
$\bt$ are given, for $0<t_1<\ldots<t_k<T$ and $x_1,\ldots,x_k\in H$,
by the formula
\begin{align*}
 &\frac{\P(\bt(t_1)\in \d x_1,\ldots,\bt(t_k)\in \d x_k)}{m(\d
  x_1)\ldots m(\d
  x_k)}\\
&=\frac{p_{t_1}(o,x_1)p_{t_2-t_1}(x_1,x_2)\ldots
  p_{t_k-t_{k-1}}(x_{k-1},x_k)p_{T-t_k}(x_k,o)}{p_T(o,o)}\, .
\end{align*}
Let also $\rho_{T}(t)=d_H(o,\bt(t)),0\leq t\leq T$ be the ``radial part'' of
$\bt$. 

We let $\ee$ be a normalized Brownian excursion \cite[Chapter
XII.4]{revuz_continuous_1991}.
% The simplest way to
% define it is as follows: let $(W_t,t\geq 0)$ be a standard
% $1$-dimensional Brownian motion, and let $g=\sup\{t<1:W_t=0\}$ and
% $d=\inf\{t>1:W_t=0\}$. Then the rescaled process
% $$\frac{|W_{g+t(d-g)}|}{\sqrt{d-g}}\, ,\qquad 0\leq t\leq 1\, ,$$
% has same distribution as $\ee$. 
Using analytical expressions for the heat kernel in $H$ and stochastic
differential equations techniques, Bougerol and Jeulin
\cite{bougerol_brownian_1999} proved the following limit theorem.

\begin{thm}[\cite{bougerol_brownian_1999}]
  \label{sec:introduction-4}
One has the convergence in distribution 
$$\left(\frac{\rho_{T}(Tt)}{\sqrt{T}},0\leq t\leq 1\right)\overset{(d)}{\underset{T\to\infty}{\longrightarrow}}(\ee_t,0\leq
t\leq 1)\, ,$$
for the uniform topology on the space $\mathcal{C}([0,1],\R)$ of
continuous functions on $[0,1]$. 
\end{thm}

\subsection{Main results}\label{sec:main-results}

Our main result, Theorem \ref{sec:introduction-3} below, gives a
geometric interpretation of Theorem \ref{sec:introduction-4}. Recall
that the Brownian continuum random tree \cite{aldous_continuum_1991-1,aldous_continuum_1993} is a random $\R$-tree coded by
the function $\ee$. More precisely, setting
$$d_\ee(s,t)=\ee_s+\ee_t-2\inf_{s\wedge t\leq u\leq s\vee t}\ee_u\,
,\qquad s,t\in [0,1]$$ defines a pseudo-distance on $[0,1]$, and we
let $(\TT_\ee=[0,1]/\{d_\ee=0\},d_\ee)$ be the quotient metric space
naturally associated with it. This space is called the Brownian
continuum random tree. We naturally distinguish the point
$o_\ee=p_\ee(0)$, where $p_\ee$ is the canonical projection, and will
usually write $\TT_\ee$ instead of $(\TT_\ee,d_\ee,o_\ee)$. This
random metric space (or more precisely its isometry class) appears as
the universal scaling limit of many tree-like random objects that
naturally appear in combinatorics and probability, see for instance
\cite{le_gall_random_2005} for a survey, and
\cite{caraceni_scaling_2014,curien_crt_2015,kortchemski_invariance_2012,panagiotou_scaling_2015,panagiotou_scaling_2014,pitman_schros_2015,rizzolo_scaling_2015,stufler_scaling_2015}
for some recent developments on the topic. Here we show that the CRT
also appears naturally in this more geometric context.

An important property of $\TT_\ee$ is the following strong re-rooting
invariance property, first noted in \cite[Proposition
4.9]{marckert_limit_2006}. For every $s,t\in [0,1]$, if we let
$s\oplus t=s+t$ if $s+t< 1$ and $s\oplus t=s+t-1$ if $s+t\geq 1$, then
for every $t\in [0,1]$
\begin{equation}
  \label{eq:3}
  (d_\ee(s\oplus t,s'\oplus t))_{s,s'\in [0,1]}\overset{(d)}{=}d_\ee\,
.
\end{equation}
This rougly says that $\TT_\ee$ pointed at $p_\ee(t)$ rather than $o_\ee$ has the same
distribution as $\TT_\ee$. 

Let $\RR_{T}=\{\bt(t):0\leq t\leq T\}\subset H$ be the range of the
Brownian loop $\bt$. We view it as a pointed metric space by
endowing it with the restriction of the hyperbolic metric
$d_H/\sqrt{T}$ renormalized by $\sqrt{T}$, and by
pointing it at $o$. As such, like $\TT_\ee$, it can be seen as a
random element of the space $\M$ of isometry classes of
pointed compact metric spaces (where two pointed metric spaces
$(X,d_X,x),(Y,d_Y,y)$ are called isometric if there exists an isometry
$\phi:X\to Y$ from $X$ onto $Y$ such that $\phi(x)=y$). This space is
equipped with the pointed Gromov-Hausdorff distance
\cite[Chapter I.5]{bridson_metric_1999}. 

\begin{thm}
  \label{sec:introduction-3}
One has the following convergence in distribution in $\M$: 
$$
\left(\RR_{T},\frac{d_H}{\sqrt{T}},o\right)\overset{(d)}{\underset{T\to\infty}{\longrightarrow}}
(\TT_\ee,d_\ee,o_\ee)\, .$$
\end{thm}

In fact, we will show that this convergence holds jointly with that of
Theorem \ref{sec:introduction-4}, meaning that 
\begin{equation}
  \label{eq:4}
  \left(\left(\frac{\rho_{T}(Tt)}{\sqrt{T}},0\leq t\leq
    1\right),\left(\RR_{T},\frac{d_H}{\sqrt{T}},o\right)\right)\overset{(d)}{\underset{T\to\infty}{\longrightarrow}}\left(\ee,\TT_\ee\right)
\end{equation}
in distribution in the product topology of
$\mathcal{C}([0,1],\R)\times \M$.

A couple of comments on Theorem \ref{sec:introduction-3} are in order. First,
it is relatively natural to see a tree structure arise in this
context, due to the fact that hyperbolic spaces can be seen as
``fattened'' trees. On the other hand, one should not think that the
limiting tree naturally lives in the hyperbolic space $H$
itself. Indeed, due to the renormalization by $\sqrt{T}$ of the
distance $d_H$, one should rather imagine that the limiting CRT is a
random subset in some asymptotic cone of $H$. It is well-known that
$H$ does not admit an asymptotic cone in a conventional (pointed
Gromov-Hausdorff) sense, but that a substitute for this notion can be
made sense of using ultralimits. A related striking property, already present
in Theorem \ref{sec:introduction-4}, is that the renormalization does
not involve scaling constant depending on the dimension $d$ of $H$, so
indeed everything happens as if large hyperbolic Brownian bridges were
living in the asymptotic cone (in the generalized sense), which does
not depend on the dimension.

Note that such a generalized asymptotic cone is a very ramified
$\R$-tree (every point disconnects the tree into uncountably many
connected components) which is in a sense much too large to consider
random subsets on a mathematically sound basis, nevertheless, it is
consistent with the idea that the (minuscule) sub-region of this cone
that is explored by a very large loop should be a random $\R$-tree.
Finally, given Theorem \ref{sec:introduction-4}, it is very natural to
guess that this random tree should be the Brownian continuum random
tree.

In section \ref{sec:infin-brown-loop} below we will also prove a
result related to Theorem \ref{sec:introduction-3} dealing with the
infinite Brownian loop in hyperbolic space, which is the ``local
limit'' (with no rescaling involved) of $b_T$ as $T\to\infty$. This is
a random path taking values in $H$, and we will show that its range,
equipped with the rescaled hyperbolic distance $a\, d_H$ for $a>0$,
converges as $a\to 0$ to a non-compact version of the continuum random
tree, the so called self-similar CRT \cite{aldous_continuum_1991}. We
refer the reader to section \ref{sec:infin-brown-loop} for precise
statements and continue our discussion of Theorem
\ref{sec:introduction-3}.

\subsection{Motivation, methods and open questions}\label{sec:methods-motivation}

We will show that Theorem
\ref{sec:introduction-3} is a relatively elementary consequence of
\begin{itemize}
\item 
Theorem \ref{sec:introduction-4},
\item
the hyperbolicity of $H$ in the sense of Gromov 
\item
a natural ``re-rooting'' invariance of Brownian loops under cyclic shifts. 
\end{itemize}
The use of functional limit theorems as Theorem
\ref{sec:introduction-4} and of re-rooting invariance properties are
powerful tools in the study of random metric spaces, as exemplified by
their use in the context of random maps. Our proofs borrow ideas of
\cite{le_gall_topological_2007,le_gall_uniqueness_2013} in
particular. 

To illustrate the robustness of the method, we will avoid as much as
possible the use of specific properties of the hyperbolic spaces $H$,
besides the fact that they satisfy the above three properties. In the rest
of the paper, we will denote by $\delta$ a constant such that $H$ is
$\delta$-hyperbolic \cite{bridson_metric_1999}.  For instance,
Bougerol and Jeulin \cite{bougerol_brownian_1999} proved Theorem
\ref{sec:introduction-4} in the more general setting of non-compact
rank 1 symmetric spaces instead of the hyperbolic space, and our proof
applies almost {\it verbatim} to this situation, replacing hyperbolic
isometries used in the re-rooting Lemma \ref{sec:introduction-1} below
by a consistent choice of isometries of the symmetric space.

In a slightly different direction, Bougerol and Jeulin
\cite{bougerol_brownian_1999} also 
proved  (and also using explicit representations of the probability
densities) that the simple random walk $(S_n,n\geq 0)$ on a $k\geq
3$-regular tree $\T_k$ and conditioned to return to the origin $o$
converges after rescaling to the Brownian excursion: if $d_{\T_k}$
denotes distance in the tree, then
\begin{equation}
  \label{eq:5}
  \left(\frac{d_{\T_k}(o,S_{\lfloor 2nt\rfloor})}{\sqrt{2n}}\right)
\quad \mbox{ given }\quad \{S_0=S_{2n}=o\}
\overset{(d)}{\underset{n\to\infty}{\longrightarrow}} \ee 
\end{equation}
in
distribution in the Skorokhod space $\mathcal{D}([0,1],\R)$ of càdlàg
functions (the use of the Skorokhod space could be avoided by taking
a continuous interpolation of the distance process above between
integer times). Our methods allow to obtain that the range
$$\left(\{S_i,0\leq i\leq 2n\},\frac{d_{\T_k}}{\sqrt{2n}},o\right)\quad  \mbox{ given
}\quad \{S_0=S_{2n}=o\}
\overset{(d)}{\underset{n\to\infty}{\longrightarrow}} \TT_\ee$$ in
distribution in $\M$. This situation is in fact simpler since in this
case the hyperbolic constant is $\delta=0$ (so that we are already
dealing with a tree metric). One should only adapt the re-rooting
invariance Lemma \ref{sec:introduction-1} below by replacing the
isometries of $H$ with those of $\T_k$, making use of the fact that it
is a transitive graph. We leave details to the reader.

As we were finishing this work, we became aware of the very recent PhD
thesis of Andrew Stewart \cite{stewart_range_2016}, who provides
another proof of the result we just mentioned on the range of random
walks on $\T_k$. Stewart's methods, based on the self-similar
structure of the continuum random tree, are independent of ours and do
not rely on Bougerol and Jeulin's result. It is indeed stressed 
in Appendix B of \cite{stewart_range_2016} that Bougerol and Jeulin's result can
be used to obtain the convergence of the range, but the sketch of
proof presented there seems quite different from our approach. There is
also some overlap between conjectures made in
\cite{stewart_range_2016} and some of the comments below.

Note that the recent work by Aïdékon and de Raphélis \cite{aidekon_scaling_2015},
proving convergence of the range of a null-recurrent biased walk on a
infinite supercritical Galton-Watson tree to a Brownian forest, is in
a similar spirit to the above discussion, but where the underlying
(random) space is only supposed to be ``statistically
homogeneous''. It would be interesting to see if the methods of
\cite{aidekon_scaling_2015} can be used to extend \eqref{eq:5} with $\T_k$ replaced by
a supercritical Galton-Watson tree. In a slightly different context,
but in a very similar spirit, we also mention the work of Duquesne \cite{duquesne_continuum_2005} on
the range of barely transient random walks on regular trees. 

In fact, we expect Theorem \ref{sec:introduction-3} to hold in a much
wider context, and that the emergence of the Brownian continuum random
tree as a limit of large Brownian loops is a signature of non-compact,
negatively curved spaces that are ``close to homogeneous''. The
intuition behind this result comes from the recent advances
\cite{gouezel_random_2013,gouezel_local_2014} on local limit theorems
for transition probabilities in hyperbolic groups. Namely, Gouëzel's
results in \cite{gouezel_local_2014} imply in particular that if $G$
is a nonelementary Gromov-hyperbolic group, and if $S$ is a finite
symmetric subset of generators of $G$, then the number $C_n$ of closed
paths of length $n$ in the Cayley graph of $G$ associated with $S$ is
asymptotically $$C_n\sim \alpha\, \beta^n\, n^{-3/2}$$ (modulo the
usual periodicity caveat) for some $\alpha=\alpha(G,S)\in (0,\infty)$
and $\beta=\beta(G,S)\in (1,\infty)$. Note that, contrary to
$\alpha,\beta$ which depend on $G,S$ the exponent $-3/2$ is
universal. In enumerative combinatorics, this kind of asymptotics is a
distinctive signature of tree structures
\cite{drmota_random_2009}. This is a first hint that a walk in $G$
conditioned to come back at its starting point after $n$ steps might
approximate a tree in some sense. In fact, this general idea is
present in the approach of \cite{gouezel_local_2014}.

However, besides these rough ideas, it is a challenge to prove a
result such as Theorem \ref{sec:introduction-4} (or the weaker Theorem
\ref{sec:introduction-3}) in contexts where strong analytical or
combinatorial tools, such as those used in
\cite{bougerol_brownian_1999}, are not available.

The proof of Theorem \ref{sec:introduction-3} will be shown in Section
\ref{sec:convergence} below, after two preliminary sections
respectively on the re-rooting invariance and on a key tightness
estimate. Finally, Section
\ref{sec:infin-brown-loop} is dedicated to study of the renormalized
infinite Brownian loop.

\bigskip

\noindent{\bf Acknowledgements: }Thanks are due to Sébastien Gouëzel
for motivating discussions, and to \'Etienne Ghys for encouragements. 

\section{Invariance under re-rooting}\label{sec:invariance-under-re}

If $\phi:H\to H$ is an isometry, then $p_t(\phi(x),\phi(y))=p_t(x,y)$,
as follows from invariance properties of the heat kernel on hyperbolic spaces
(\cite{anker-ostellari,grigoryan-noguchi}).  From there, a natural property of invariance under
cyclic shifts holds.  For $x\in H, x\neq o$, we let
$\phi_x:H\to H$ be the unique hyperbolic isometry sending $x$
to $o$, and let $\phi_o$ be the identity map.

\begin{lmm}
  \label{sec:introduction-1}
Fix $T>0$ and $t\in [0,T]$. Then the processes 
$$\bt\, ,\qquad  
\phi_{\bt(t)}(\bt(\cdot+t\! \mod T))\, ,$$ have same
distribution. Here, by convention, we let $s+t\!\mod T$ be the unique
representative in $[0,T)$. 
\end{lmm}

\begin{proof}
For convenience, let $X_s:=\phi_{b_T(t)}(b_T(s+t \!\mod T))$ for
$s\in[0,T]$. Let $s,r\in (0,T)$ with $s<r$ and $F: H^2\rightarrow\mathbb{R}_+$
be a measurable function. We will show that
 \[
 \E\left(F(X_s,X_r)\right)=\E\left(F(b_T(s),b_T(r))\right)
\]
We prove it in the case where $s<T-t<r<T$, the situations where
$s<r<T-t$ and $T-t<s<r$ are easier and left to the reader.  Observe
that $0<t+r-T<t<t+s<T$, so by the finite-dimensional distribution of
$b_T$, we have
\begin{align*}
 &\E\left(F(X_s,X_r)\right)=\E\left[F\Big(\phi_{b_T(t)}(b_T(s+t)), \phi_{b_T(t)}(b_T(t+r-T))\Big)\right]\\
 =&\int_H m(\d x_1)\int_H m(\d x_2)\int_H m(\d x_3)\\
&\qquad \times \frac{p_{t+r-T}(o, x_1)p_{T-r}(x_1,x_2)p_{s}(x_2,x_3)p_{T-t-s}(x_3,o)}{p_T(o,o)}F(\phi_{x_2}(x_3),\phi_{x_2}(x_1))
\end{align*}
Since $p_t$ is invariant by the isometry $\phi_{x_2}$, we deduce that
\begin{align*}
&\E\left(F(X_s,X_r)\right)=\int_H m(\d x_2)\int_{H^2}m(\d x_1)m(\d x_3)\\
&\times \frac{p_{t+r-T}(\phi_{x_2}(o), \phi_{x_2}(x_1))p_{T-r}(\phi_{x_2}(x_1),o)p_{s}(o,\phi_{x_2}(x_3))p_{T-t-s}(\phi_{x_2}(x_3), \phi_{x_2}(o))}{p_T(o,o)}F(\phi_{x_2}(x_3),\phi_{x_2}(x_1))
\end{align*}
Let us write $y_1=\phi_{x_2}(x_3)$,$y_2=\phi_{x_2}(x_1)$, and
$x_2'=\phi_{x_2}(o)$. Note that in the Poincaré ball model of $H$ with
origin $o=0\in \R^d$, $x'_2$ is simply the point $-x_2$, so that
clearly $\int_H f(x'_2)m(\d x_2)=\int_H f(x_2)m(\d x_2)$ for every
non-negative measurable $f$.  It follows that
\begin{align*}
&\E\left(F(X_s,X_r)\right)\\
=&\int_H m(\d x_2)\int_{H^2}m(\d y_1)m(\d y_2)\frac{p_{t+r-T}(x'_2, y_2)p_{T-r}(y_2,o)p_s(o,y_1)p_{T-t-s}(y_1,x'_2)}{p_T(o,o)}F(y_1,y_2)\\
=&\int_{H^2}m(\d y_1)m(\d
y_2)\frac{p_s(o,y_1)p_{T-r}(y_2,o)}{p_T(o,o)}F(y_1,y_2)\int_H m(\d
x_2) p_{T-t-s}(y_1,x'_2)p_{t+r-T}(x'_2, y_2)\\
&\int_{H^2}m(\d y_1)m(\d y_2)\frac{p_s(o,y_1)p_{T-r}(y_2,o)}{p_T(o,o)}F(y_1,y_2)\int_H m(\d x_2) p_{T-t-s}(y_1,x_2)p_{t+r-T}(x_2, y_2)\\
=&\int_{H^2}m(\d y_1)m(\d y_2)\frac{p_s(o,y_1)p_{T-r}(y_2,o)}{p_T(o,o)}F(y_1,y_2) p_{r-s}(y_1,y_2)\\
=&\E\left(F(b_T(s),b_T(r))\right)
\end{align*}
The proof of the equality of all finite-dimensional marginals is 
similar to this case, with longer formulas, and we leave it as an
exercise to the reader. 
This concludes Lemma \ref{sec:introduction-1} because of the continuity of $b$.
\end{proof}

Note that it does not really matter which isometry sending $\bt(t)$ to
$o$ we choose (we could also have chosen the unique parabolic isometry
sending $x$ to $o$) but of course one should perform this choice in a
consistent way in such a way that the image of $m$ under $x\mapsto
\phi_x(o)$ is $m$ (note that this image is clearly invariant under the
action of isometries of $H$ so is a constant multiple $c\, m$ of $m$,
and noting that $x\mapsto \phi_x(o)$ is a measurable involution from
$H$ onto $H$, this entails that $c=1$.)

\section{Tightness estimate}\label{sec:tightness-estimate}

For $T,\eta>0$, we let $\NN(T,\eta)$ be the minimal number of balls of
radius $\eta$ (with respect to the metric $d_H/\sqrt{T}$) necessary to
cover the range $\RR_{T}$: 
$$\NN(T,\eta)=\inf\left\{N\geq 1:\exists x_1,\ldots,x_N\in
  H,\RR_{T}\subset\bigcup_{k=1}^NB_{d_H}(x_k,\eta\sqrt{T})\right\}\, .$$

\begin{lmm}
  \label{sec:tightness-estimate-1}
It holds that for every $N\geq 2$ and $\eta>0$, 
$$\limsup_{T\to\infty}\,
\P\left(\RR_{T}\not\subset\bigcup_{i=0}^{N-1}B_{d_H}(\bt(Ti/N),\eta\sqrt{T})\right)\leq
\frac{12}{\eta} \sqrt{\frac{N}{\pi}}\, e^{-\eta^2 (N-1)/18}\, ,$$
and in particular one has
$\lim_{N\to\infty}\limsup_{T\to\infty}\P(\NN(T,\eta)>N)=0$. 
\end{lmm}

\begin{proof}
By the union bound and the
  re-rooting lemma \ref{sec:introduction-1}, 
  \begin{align*}
 \lefteqn{\P\left(\RR_{T}\not\subset\bigcup_{i=0}^{N-1}B_{d_H}(\bt(Ti/N),\eta\sqrt{T})\right)}\\
&\leq
\sum_{i=0}^{N-1}\P\left(\sup\left\{d_H(\bt(iT/N),\bt((s+i/N)T)):s\in[0,1/N]
\right\}\geq \eta\sqrt{T}\right)\\
&=N\, \P\left(\sup\left\{d_H(o,\bt(Ts)):s\in[0,1/N]
\right\}\geq \eta\sqrt{T}\right). 
  \end{align*}
Theorem \ref{sec:introduction-4} implies that 
$$\limsup_{T\to\infty}\P\left(\sup\left\{d_H(o,\bt(Ts)):s\in[0,1/N]
  \right\}\geq \eta\sqrt{T}\right)\leq \P(\sup_{[0,1/N]}\ee>\eta)\,
.$$ To bound this probability, one can use for instance the fact (see
Theorem XII.4.2 and Exercise XI.3.6 in \cite{revuz_continuous_1991})
that $((1-s)X_{s/(1-s)},0\leq s\leq 1)$ has same distribution as $\ee$
if $X$ is a $3$-dimensional Bessel process. This shows that
$\sup_{[0,1/N]}\ee$ is stochastically dominated by
$\sup_{[0,1/(N-1)]}X$. Then one can use the fact that $X$ has same
distribution as the Euclidean norm of a standard $3$-dimensional
Brownian motion. Using this, we easily get
$$\P\left(\sup_{[0,1/(N-1)]}X>\eta\right)\leq 6\,
\P\left(\sup \{W_s:0\leq s\leq 1/(N-1)\}>\eta/3\right)\, ,$$ where $(W_t,t\geq
0)$ is a standard Brownian motion in $\R$. Using the fact that $\sup
\{W_s:0\leq s\leq t\}$ has same distribution as $|W_t|$ and the
estimate $\P(|W_1|\geq x)\leq 2\exp(-x^2/2)/x\sqrt{2\pi}$, we get the
wanted bound.  We
conclude since clearly $\NN(T,\eta)>N$ implies that
$\RR_{T}\not\subset\bigcup_{i=0}^{N-1}B_{d_H}(\bt(Ti/N),\eta\sqrt{T})$.
\end{proof}

\bigskip

A crucial corollary of the tightness estimate (and hyperbolicity) is
the fact that the range $\RR_{T}$ cannot avoid large portions of
geodesics between the points it visits. For $x,y\in H$, let $[x,y]$ be
the (hyperbolic) geodesic segment between $x$ and $y$. For $0\leq
s\leq t\leq T$ we let $\RR_{T}(s,t)=\{\bt(u):s\leq u\leq t\}$, and
for $r>0$ we define the event
$$\Lambda_{T}(r)=\left\{\exists\, s\leq t\in [0,T]:\sup_{y\in
    [\bt(s),\bt(t)]}d_H(y,\RR_{T}(s,t))\geq
  r\right\}\, .$$

\begin{lmm}
  \label{sec:tightness-estimate-2}
For every $\eta>0$, one has
$\P(\Lambda_{T}(\eta\sqrt{T}))\to 0$ as $T\to\infty$. 
\end{lmm}

\begin{proof}
  A standard property of $\delta$-Gromov-hyperbolic spaces (see
  Proposition III.1.6 in \cite{bridson_metric_1999}) is that if $c$ is a continuous
  path that avoids a ball $B(z,r)$ around some vertex $z$ on a
  geodesic between the endpoints of $c$, then $c$ must be of length at
  least $2^{(r-1)/\delta}$. This implies the property that if moreover
  we assume that $c$ avoids the larger ball $B(z,2r)$, then its image
  cannot be covered by less than
  $2^{(r-1)/\delta}/(2r)$ balls of radius $r$: otherwise, by
  possibly modifying the path $c$ by a piecewise geodesic path inside
  each ball of a cover of the image of $c$ by balls of radius $r$, we
  would find a path that avoids $B(z,r)$ but is of length at most
  $2r\times 2^{(r-1)/\delta}/(2r)$, a contradiction.

  On the event $\Lambda_{T}(2\eta\sqrt{T})$, there exist $s<t$ in
  $[0,1]$ and $y\in [\bt(Ts),\bt(Tt)]$ such that
  $d_H(y,\RR_{T}(s,t))\geq 2\eta\sqrt{T}$, meaning that the portion
  of the path of $\bt$ between times $s$ and $t$ avoids
  $B_{d_H}(y,2\eta\sqrt{T})$. By the above discussion, this implies
  that
$$\NN(T,\eta)>\frac{2^{(\eta\sqrt{T}-1)/\delta}}{2\eta\sqrt{T}}\,
.$$
Since the latter lower bound diverges for any $\eta>0$ as $T\to\infty$,
we conclude immediately from Lemma  \ref{sec:tightness-estimate-1}. 
\end{proof}

We now define a continuous random function $d_{(T)}$ on $[0,1]^2$ by the
formula
$$d_{(T)}(s,t)=\frac{d_H(b_T(Ts),b_T(Tt))}{\sqrt{T}}\, ,\qquad
0\leq s,t\leq 1\, .$$

\begin{lmm}
  \label{sec:tightness-estimate-3}
The family of laws of $d_{(T)}$, for $T\geq 1$, is relatively compact for
the weak topology on probability measures on
$\mathcal{C}([0,1]^2,\R)$. 
\end{lmm}

\begin{proof}
  Note that for every $s,s',t,t'\in [0,1]$, one has, by the triangle
  inequality, 
$$|d_{(T)}(s,t)-d_{(T)}(s',t')|\leq d_{(T)}(s,s')+d_{(T)}(t,t')\, .$$
This shows that the modulus of continuity of $d_{(T)}$ is bounded as
follows: for $\alpha>0$, 
$$\sup_{\substack{|s-s'|\leq \alpha\\ |t-t'|\leq
    \alpha}}|d_{(T)}(s,t)-d_{(T)}(s',t')|\leq 2\sup_{|s-s'|\leq
  \alpha}d_{(T)}(s,s')\, .$$
Now, for every $\eta>0$, we obtain 
\begin{equation}
  \label{eq:10}
  \P\Bigg(\sup_{\substack{|s-s'|\leq \alpha\\ |t-t'|\leq
    \alpha}}|d_{(T)}(s,t)-d_{(T)}(s',t')|\geq 8\eta\Bigg)\leq \P\left(\sup_{|s-s'|\leq
  \alpha}d_{(T)}(s,s')\geq 4\eta\right)\, .
\end{equation}
By $\delta$-hyperbolicity, for every
 $a,b,c\in H$, it holds that  
$$2d_H(a,[b,c])+d_H(b,c)\leq d_H(a,b)+d_H(a,c)+4\delta\, ,$$
see (8.4)
  in \cite{burago_course_2001}. We apply this to
  $a=o,b=\bt(Ts),c=\bt(Ts')$ for some $s\leq s'$, so that, if we let $y\in [\bt(Ts),\bt(Ts')]$ be such that
$d_H(o,y)=d_H(o,[\bt(Ts),\bt(Ts')])$,
$$2\, d_H(o,y)+d_H(\bt(Ts),\bt(Ts'))\leq
d_H(o,\bt(Ts))+d_H(o,\bt(Ts'))+4\delta\, .$$
  Outside the event
$\Lambda_{T}(\eta\sqrt{T})$, we can find $u\in [s,s']$ such that
$d_H(\bt(Tu),y)\leq \eta\sqrt{T}$, so that 
\begin{align*}
  d_H(\bt(Ts),\bt(Ts'))&\leq
  \rho_{T}(Ts)+\rho_{T}(Ts')-2\rho_{T}(Tu)+4\delta+2\eta\sqrt{T}\\
&\leq   \rho_{T}(Ts)+\rho_{T}(Ts')-2\, \inf_{v\in
  [s,s']}\rho_{T}(Tv)+4\delta+2\eta\sqrt{T}\, , 
\end{align*}
which, by letting  $\rho_{(T)}=\rho_T(T\cdot)/\sqrt{T}$, can be
rewritten as
\begin{equation}
  \label{eq:11}
  d_{(T)}(s,s')\leq \rho_{(T)}(s)+\rho_{(T)}(s')-2\, \inf_{v\in
  [s,s']}\rho_{(T)}(v)+\frac{4\delta}{\sqrt{T}}+2\eta\, .
\end{equation}
Hence, we have proved that outside $\Lambda_T(\eta\sqrt{T})$, we have
$$\sup_{|s-s'|\leq
  \alpha}d_{(T)}(s,s')\leq
2\omega(\rho_{(T)},\alpha)+\frac{4\delta}{\sqrt{T}}+2\eta\, ,$$
where $\omega(f,\cdot)$ denotes the
modulus of continuity of the function $f$. 
Therefore, 
$$\P\left(\sup_{|s-s'|\leq
  \alpha}d_{(T)}(s,s')\geq 4\eta\right)\leq
\P(\Lambda_T(\eta\sqrt{T}))+\P\left(
\omega(\rho_{(T)},\alpha)
\geq \eta- \frac{2\delta}{\sqrt{T}}
\right)$$

By \eqref{eq:10}, Theorem \ref{sec:introduction-4} and Lemma
\ref{sec:tightness-estimate-2}, we conclude that 
$$\limsup_{T\to\infty} \P\Bigg(\sup_{\substack{|s-s'|\leq \alpha\\ |t-t'|\leq
    \alpha}}|d_{(T)}(s,t)-d_{(T)}(s',t')|\geq 8\eta\Bigg)\leq
\P\left(\omega(\ee,\alpha)\geq \eta\right)\, ,
$$
and this converges to $0$ as $\alpha\to0$ for any fixed value of
$\eta$. Together with the fact that $d_{(T)}(0,0)=0$, this allows to
conclude by standard results \cite{billingsley_convergence_1999}. 
\end{proof}

\section{Convergence}\label{sec:convergence}

In this section, we finish the proof of Theorem
\ref{sec:introduction-3}. 

\begin{lmm}
  \label{sec:convergence-1}
It holds that 
\begin{equation}
  \label{eq:1}
\left(\rho_{(T)},d_{(T)}\right)
\overset{(d)}{\underset{T\to\infty}{\longrightarrow}}(\ee,d_\ee)
\end{equation}
in distribution in $\mathcal{C}([0,1],\R)\times \mathcal{C}([0,1]^2,\R)$. 
\end{lmm}

\begin{proof}
  By Prokhorov's Theorem, based on Theorem \ref{sec:introduction-4}
  and Lemma \ref{sec:tightness-estimate-3}, the laws of the random
  variables in the left-hand side of \eqref{eq:1} form a relatively
  compact family of probability measures
  on $\mathcal{C}([0,1],\R)\times \mathcal{C}([0,1]^2,\R)$. We
  deduce that for any sequence $T_n\to\infty$, we can extract a
  subsequence along which the pair of random variables in
  \eqref{eq:1} converges in distribution towards a certain limiting
  random variable
$(\ee,d)$. 
The slight abuse of notation in denoting the first
 component by $\ee$ is motivated by the fact that its
 marginal law is that of the normalized Brownian excursion, by Theorem
 \ref{sec:introduction-4}. By using Skorokhod's theorem, we may and
 will assume that the convergence holds in the almost sure sense, which will
 simplify some of the arguments to come. 

 To conclude, it suffices to show that $d=d_\ee$ a.s., since this will
 characterize uniquely the limiting distribution, hence allowing to
 obtain the convergence result without having to take extractions.
 Note that $d(0,s)=\ee_s=d_\ee(0,s)$ for every $s\in [0,1]$ almost
 surely, since $d_{(T)}(0,s)=\rho_{(T)}(s)$ and by passing to the
 limit.  But the re-rooting Lemma \ref{sec:introduction-1} implies
 that $d_{(T)}(s,t)$ has same distribution as $d_{(T)}(0,t-s)$ for
 every $s\leq t$ in $[0,1]$. By passing to the limit, we thus see that
 $d(s,t)$ has same distribution as $d(0,t-s)=d_\ee(0,t-s)$. Using the
 re-rooting invariance of the Brownian continuum random tree
 \eqref{eq:3}, we obtain that in turn, this has same distribution as
 $d_\ee(s,t)$.  On the other hand, taking the limit in \eqref{eq:11}
 (and using Lemma \ref{sec:tightness-estimate-2}) shows that
 $d(s,t)\leq d_\ee(s,t)$ almost surely. Therefore, equality must hold
 almost surely, because the expectation of the (nonnegative)
 difference is $0$.
\end{proof}

It is now straightforward to conclude the proof of \eqref{eq:4}, hence
of Theorem \ref{sec:introduction-3}. Still assuming that the
convergence \eqref{eq:1} holds almost surely, the set
$\{(b_T(sT),p_\ee(s)):s\in [0,1]\}$ defines a
correspondence \cite[Section 7.3.3]{burago_course_2001} between
$\RR_T$ and $\TT_\ee$ containing $(o,o_\ee)$, and of distortion
bounded above by
$$\sup_{s,t\in
  [0,1]}|d_{(T)}(s,t)-d_\ee(s,t)|\underset{T\to\infty}{\longrightarrow}
0\, ,\qquad \mbox{a.s.}$$
This shows that the pointed Gromov-Hausdorff distance between
$(\RR_T,d_H/\sqrt{T},o)$ and $\TT_\ee$ converges to $0$ almost surely, as wanted.

\section{The infinite Brownian loop and the
  self-similar CRT}\label{sec:infin-brown-loop}

We now argue that our methods also allow to prove a result related to
Theorem \ref{sec:introduction-3}, which deals with the so-called {\em
  infinite Brownian loop}. The latter can be obtained as a local limit
of large Brownian loops. Specifically, let us extend the bridge
$b_T$ by $T$-periodicity and view it a random function $(b_T(t),t\in \R)$. We
equip the space $\mathcal{C}(\R,H)$ with the compact-open topology, so
that convergence in this space is equivalent to uniform convergence
over compact intervals.

An important result by Anker, Bougerol and Jeulin \cite[Theorem 1.2,
Proposition 2.6 and Proposition 4.2]{anker_infinite_2002} implies that
in every non-compact symmetric space $H$, as $T\to\infty$, the
Brownian bridge $b_T$ converges in distribution in $\mathcal{C}(\R,H)$
towards a limit $b_\infty$, called the infinite Brownian loop. As
before in this paper, we will only focus on the case where $H$ is the
hyperbolic space, which corresponds to rank~$1$ symmetric spaces.

Anker, Bougerol and Jeulin further show the following result. Let
$\rho_\infty(t)=d_H(o,b_\infty(t))$ for $t\in \R$. Theorems 1.4, 1.5
and 7.1 (iii) in \cite{anker_infinite_2002}, again in the very special case
of rank~$1$ symmetric spaces, can be stated as follows. 

\begin{thm}
  \label{sec:infin-brown-loop-2}
  Let $R,R'$ be two independent Bessel processes of dimension $3$
  started from $0$, and let $X_t=R_t$ if $t\geq 0$, $X_t=R'_{-t}$ if
  $t<0$. Then it holds that
  \begin{equation}
    \label{eq:6}
\left( a\,\rho_\infty(t/a^2),t\in \R\right)\overset{(d)}{\underset{a\to 0}{\longrightarrow}}X\, ,
  \end{equation}
in distribution for the compact-open topology on $\mathcal{C}(\R,\R)$.
\end{thm}

From the process $X$, we can build a locally compact pointed random
metric space called the self-similar Brownian continuum random tree
\cite{aldous_continuum_1991}, in a similar way to Section
\ref{sec:main-results}. Namely, we define a pseudo-distance $d_X$ on
$\R$ by the formula
$$d_X(s,t)=X_s+X_t-\check{X}(s,t)\, ,$$
where $\check{X}(s,t)=\inf_{s\wedge t\leq u\leq s\vee t}X_u$ whenever
$st\geq 0$, and $\check{X}(s,t)=\inf_{u\notin[s\wedge t,s\vee t]}X_u$
otherwise. We let $\TT_X=(X/\{d_X=0\},d_X,o_X)$ be the quotient
metric space, pointed at $o_X=p_X(0)$ where $p_X$ is the canonical
projection. This defines a locally compact, complete pointed
$\R$-tree. 

We let $\RR_\infty=\{b_\infty(t),t\in \R\}$ be the range of
$b_\infty$, which we canonically view as the pointed metric space
$(\{b_\infty(t),t\in \R\}, d_H,o)$. We use the notation $a\, M=(M, a
d,x)$ whenever $(M,d,x)$ is a pointed metric space and $a>0$.

\begin{thm}
  \label{sec:infin-brown-loop-3}
It holds that 
$$a\, \RR_\infty\overset{(d)}{\underset{a\to 0}{\longrightarrow}}
\TT_X\, ,$$
in distribution for the local Gromov-Hausdorff topology. This
convergence holds jointly with \eqref{eq:6}. 
\end{thm}

This result can be obtained by adapting our arguments, but since we
are now dealing with local Gromov-Hausdorff convergence \cite[Chapter
8.1]{burago_course_2001}, which (very) roughly speaking amounts to the
Gromov-Hausdorff convergence of balls centered at the distinguished
point, some extra care should be taken. 

\subsection{Basic properties of $\TT_X$}

Let us gather some of the important properties of the self-similar
CRT. First, it also satisfies a property of invariance under
re-rooting that will be crucial to us. Here and below, the set
$\mathcal{C}(\R^2,\R)$ will be endowed with the compact-open topology.

\begin{prp}
  \label{sec:infin-brown-loop-8}For every $t\in \R$, the random
  function $(d_X(s+t,s'+t))_{s,s'\in \R}$ in $\mathcal{C}(\R^2,\R)$
  has same distribution as $d_X$.
\end{prp}

\begin{proof}
  This can be shown from the re-rooting invariance of the CRT, by a
  limiting argument. However, some care has to be taken. Let
  $\ee^\lambda_t=\sqrt{\lambda}\, \ee(t/\lambda),0\leq t\leq \lambda$ be the Brownian
  excursion with duration $\lambda$. We let
  $d^\lambda_\ee(s,t)=\sqrt{\lambda}\, d_\ee(s/\lambda,t/\lambda)$, 
  defining a random pseudo-distance on $[0,\lambda]$. By \cite[Proposition
  3]{curien_brownian_2014}, for any $A\in (0,\lambda/2)$, the triplet 
$$\left((\ee^\lambda_t)_{0\leq t\leq A},(\ee^\lambda_{\lambda-t})_{0\leq
    t\leq A},\min_{A\leq t\leq \lambda-A}\ee^\lambda_t\right)$$ is
absolutely continuous with respect to the law of $((X_t)_{0\leq t\leq
  A},(X_{-t})_{0\leq t\leq A},\check{X}(-A,A))$, with a density
$\Delta_{\lambda,A}(\omega(A),\omega'(A),z)$ such that
$\Delta_{\lambda,A}(x,y,z)$ converges to $1$ as $\lambda\to\infty$
whenever $0<z< x\wedge y$. Therefore, for every $\eps,A>0$, there
exists $\lambda_0=\lambda_0(\eps,A)>2A$ and a coupling of
$\ee^\lambda$ and $X$ on some probability space such that for every
$\lambda\geq \lambda_0$, outside an event
$\mathcal{A}=\mathcal{A}(\eps,A)$ of probability at most $\eps$, we
have
$$\ee^\lambda_t=X_t\, ,\quad \ee^\lambda_{\lambda-t}=X_{-t}\quad\mbox{
  for }t\in [0,A]\, ,\qquad \mbox{and } \quad\min_{A\leq t\leq
  \lambda-A}\ee^\lambda_t=\check{X}(-A,A)\, .$$ In particular, still on
$\mathcal{A}^c$, it holds that for $s,s'\in [0,A]$,
$$d^\lambda_\ee(s,s')=d_X(s,s')\, ,\quad
d^\lambda_\ee(\lambda-s,\lambda-s')=d_X(-s,-s')\, ,\quad
d^\lambda_\ee(s,T-s')=d_X(s,-s')\, .$$
Defining $\tilde{\ee}^\lambda_t=\ee^\lambda_t$ for $t\in [0,\lambda/2]$
and $\tilde{\ee}^\lambda_t=\ee^\lambda_{\lambda+t}$ for $t\in
[-\lambda/2,0]$, we let 
$$\tilde{d}^\lambda_\ee(s,t)=\tilde{\ee}^\lambda_s+\tilde{\ee}^\lambda_t-2\check{\ee}^\lambda(s,t)\,
,\qquad s,t\in [-\lambda/2,\lambda/2]$$ where $\check{\ee}^\lambda(s,t)=\min_{s\wedge t\leq u\leq s\vee
  t}\tilde{\ee}^\lambda_u$ if $st\geq 0$, and $\min_{u\in
  [-\lambda/2,s\wedge t]\cup [s\vee
  t,\lambda/2]}\tilde{\ee}^\lambda_u$ otherwise. Then on the coupling
event $\mathcal{A}^c$, one has $\tilde{d}_\ee(s,s')=d_X(s,s')$ for
every $s,s'\in [-A,A]$. 

The re-rooting invariance for $d_{\ee}$ together with the
definition of $\ee^\lambda$ shows that if $s\oplus_\lambda t$ denotes the
representative of $s+t$ modulo $\lambda$ in the interval
$[-\lambda/2,\lambda/2)$, then
$(\tilde{d}^\lambda_{\ee}(s\oplus_\lambda t,s'\oplus_\lambda t),s,s'\in
[-\lambda/2,\lambda/2])$ has same distribution as
$\tilde{d}^\lambda_{\ee}$. Fixing the value of $t$ and fixing
$A>2|t|$, for 
$\lambda\geq \lambda_0$, we see that $s\oplus_\lambda t=s+t\in [-A,A]$ for every
$s\in [-A/2,A/2]$, so on the coupling event $\mathcal{A}^c$
$$(\tilde{d}^\lambda_{\ee}(s\oplus_\lambda t,s'\oplus_\lambda
t))_{s,s'\in [-A/2,A/2]}=(d_X(s+t,s'+t))_{s,s'\in [-A/2,A/2]}$$ while
this has same law as the restriction of $\tilde{d}^\lambda_\ee$ to
$[-A/2,A/2]$. Since the left-hand side has same distribution as
$\tilde{d}^\lambda_\ee$ restricted to $[-A/2,A/2]^2$, and that the
latter is equal on the coupling event $\mathcal{A}^c$ to the restriction
of $d_X$ to $[-A/2,A/2]^2$, we see for every $A,\eps>0$, the total
variation distance between the laws of $(d_X(t+s,t+s'))_{s,s'\in
  [-A/2,A/2]}$ and $(d_X(s,s'))_{s,s'\in [-A/2,A/2]}$ is at most
$2\P(\mathcal{A}(\eps,A))\leq 2\eps$. Since $\eps$ is arbitrary, we
see that these laws are equal, and since $A$ is arbitrary as well, we
can conclude.
\end{proof}

The next property is a geometric property, which is often referred to
as the fact that $\TT_X$ has ``a unique infinite spine'', also called
``baseline'' in \cite{aldous_continuum_1991}. Recall that
a geodesic ray in a length metric space is a subset that is isometric
to $\R_+$ (that is identified with its natural parametrization by
$\R_+$). This fact is essentially a consequence of the way it is
introduced in \cite{aldous_continuum_1991}, but it is also easy to
prove it directly from the above definition, and we leave it as an
exercise. 

\begin{prp}
\label{sec:basic-prop-tt_x-1}
Almost surely, $\TT_X$ has a unique geodesic ray
starting from $o_X$. 
\end{prp}

It is not difficult to see that if we let 
\begin{equation}
  \label{eq:12}
\Gamma_+(r)=\sup\{t\geq 0:X_t=r\}\, ,\qquad \Gamma_-(r)=\inf\{t\leq
0:X_t=r\}\, ,\qquad r\geq 0\, ,
\end{equation}
then the unique geodesic ray of the last proposition is $
p_X(\Gamma_+(r))=p_X(\Gamma_-(r)),r\geq 0$, which means
that 
\begin{equation}
  \label{eq:13}
  d_X(\Gamma_\pm(r),\Gamma_\pm(r'))=|r-r'|\quad \mbox{ for every
  }\quad r,r'\geq
0\, .
\end{equation}

\subsection{Basic properties of $\RR_\infty$}

We now discuss the results of Sections \ref{sec:invariance-under-re}
and \ref{sec:tightness-estimate} that are easily generalized to the
infinite loop. The re-rooting Lemma \ref{sec:introduction-1}
generalizes indeed, by a simple passage as $T\to\infty$ that we leave
as an exercise to the reader.

\begin{lmm}
  \label{sec:infin-brown-loop-4}
For every $t\in \R$, the processes $b_\infty$ and
$\phi_{b_\infty(t)}(b_\infty(\cdot+t))$ have the same distribution. 
\end{lmm}

The tightness estimate of Lemma 4 does not generalize {\it verbatim},
but should be adapted in the following way.  For $s\leq t$, we let
$\RR_\infty(s,t)=\{b_\infty(u):s\leq u\leq t\}$. For simplicity, for
$A>0$ we let $\RR_\infty(A)=\RR_\infty(-A,A)$. As for $\RR_\infty$,
this set is canonically endowed with the restriction of $d_H$ and
pointed at $o$.

\begin{lmm}
  \label{sec:infin-brown-loop-5}
  For every integers $A,N\geq 2$ and every $\eta>0$, it holds that
$$\limsup_{a\to 0}\P\left(\RR_\infty(A/a^2)\not\subset
  \bigcup_{i=-AN}^{AN}B_{d_H}(b_\infty(i/Na^2),\eta/a)\right) \leq
\frac{25\, A}{\eta}\sqrt{\frac{N}{\pi}}e^{-\eta^2(N-1)/18}\, .$$
\end{lmm}

The proof is the same as Lemma \ref{sec:tightness-estimate-1}, using
the union bound, and then the re-rooting Lemma
\ref{sec:infin-brown-loop-4} and the convergence \eqref{eq:6}. The
following analog of Lemma \ref{sec:tightness-estimate-2} is deduced in
exactly the same way, letting
$$\Lambda_\infty(r,A)=\left\{\exists\, s\leq t\in [-A,A]:\, \sup_{y\in
    [b_\infty(s),b_\infty(t)]}d_H(y,\RR_\infty(s,t))\geq
  r\right\}\, .$$ 
We also define a distance function and a renormalized radial process
by the formula
$$d_{(a)}(s,t)=a\, d_H(b_\infty(s/a^2),b_\infty(t/a^2))\, ,\qquad
\rho_{(a)}(t)=a\, \rho_\infty(t/a^2)=d_{(a)}(0,t)\, ,$$ for every
$s,t\in \R$. These should not be mistaken for $d_{(T)},\rho_{(T)}$
used in earlier sections. We state a consequence of Lemma
\ref{sec:infin-brown-loop-5}, proved in the same way as Lemma
\ref{sec:tightness-estimate-2} and the beginning of the proof of Lemma
\ref{sec:tightness-estimate-3}. 

\begin{lmm}
  \label{sec:infin-brown-loop-6}
For every $A,\eta>0$, one has $\P(\Lambda_\infty(\eta/a,A/a^2))\to 0$ as $a\to
0$. Moreover, outside the event $\Lambda_\infty(\eta/a,A/a^2)$, one
has, for every $s\leq s'$ in the interval $[-A,A]$, 
\begin{equation}
  \label{eq:2}
  d_{(a)}(s,s')\leq \rho_{(a)}(s)+\rho_{(a)}(s')-2\inf_{v\in
    [s,s']}\rho_{(a)}(v)+4a\delta+2\eta
\end{equation}
\end{lmm}

\subsection{Convergence}\label{sec:convergence-2}

We can now state the following key lemma. 

\begin{lmm}
  \label{sec:infin-brown-loop-9}
We have the following convergence in distribution in
$\mathcal{C}(\R,\R)\times \mathcal{C}(\R^2,\R)$: 
$$(\rho_{(a)},d_{(a)})\overset{(d)}{\underset{a\to
    0}{\longrightarrow}}(X,d_X)\, .$$
\end{lmm}

\begin{proof}
  Using Lemma \ref{sec:infin-brown-loop-6} instead of Lemma
  \ref{sec:tightness-estimate-2}, we deduce exactly as in Lemma
  \ref{sec:tightness-estimate-3} that the family of laws of $d_{(a)}$
  for $a\leq 1$ is a tight family of random variables (due to the fact
  that we are considering the compact-open topology, it suffices to
  control the modulus of continuity of $d_{(a)}$ restricted to compact subsets of
  $\R^2$ of the form $[-A,A]^2$). As in the beginning of the proof of Lemma
  \ref{sec:convergence-1} it holds that for any sequence $a_n\to0$, we
  can extract a subsequence along which $(\rho_{(a)},d_{(a)})$
  converges in distribution in $\mathcal{C}(\R,\R)\times
  \mathcal{C}(\R^2,\R)$ to some limit $(X,d)$. Without loss of
  generality, we may assume that the convergence holds almost surely,
  and it remains to check that $d=d_X$ almost surely.

Note that by using Lemma \ref{sec:infin-brown-loop-4} and passing to
the limit, the function $d$ satisfies the same re-rooting invariance
property as $d_X$: namely, for every $t\in \R$, the function
$(d(s+t,s'+t),s,s'\in \R)$ has same distribution as $d$.  Moreover,
passing to the limit in \eqref{eq:2} and using Lemma
\ref{sec:infin-brown-loop-6} shows that $d(s,t)\leq
X_s+X_t-2\inf_{[s,t]} X$ for every $s\leq t$ in $\R$. However, at this
point one should note that this upper-bound is equal to $d_X(s,t)$
only if $st\geq 0$.  For such $s,t$, the rest of the argument applies
without change: assuming for instance $0\leq s\leq t$, $d(s,t)$ has
same distribution as $d(0,t-s)$ by re-rooting invariance. Then
$d(0,t-s)=X_{t-s}=d_X(0,t-s)$, which has same distribution as
$d_X(s,t)$ by the re-rooting Proposition \ref{sec:infin-brown-loop-8}. Since
$d(s,t)\leq d_X(s,t)$ almost surely, we deduce that they are in fact
equal almost surely. We obtain that the restrictions of $d$ to
$(\R_+)^2$ and $(\R_-)^2$ are respectively equal the same restrictions
of $d_X$.

By this last fact and re-rooting invariance, we obtain that for every
$A>0$,
$$(d(s,t))_{s,t\geq -A}\overset{(d)}{=} (d(s+A,t+A))_{s,t\geq
  -A}=(d_X(s+A,t+A))_{s,t\geq -A}\overset{(d)}{=}(d_X(s,t))_{s,t\geq
  -A}\, ,$$ and by letting $A\to\infty$ we obtain that $d$ has same
distribution as $d_X$. In particular, $d$ is a pseudo-distance on $\R$
such that the quotient space $\TT=(\R/\{d=0\},d,o_\TT)$ is a real tree
with the same distribution as $\TT_X$ (with $o_\TT=p(0)$ where $p$ is
canonical projection). Therefore, the uniqueness of the geodesic ray
stated in Proposition \ref{sec:basic-prop-tt_x-1} must also be true
for $\TT$. On the other hand, since the restrictions of $d$ and $d_X$ to
$\R_+^2$ and $\R_-^2$ are equal, the images by $p$ of the two
functions $\Gamma_+$ and $\Gamma_-$ of \eqref{eq:12}
are two geodesic rays $\gamma_+,\gamma_-$ from $o_\TT$. This is due to
the fact that these
functions take values in $\R_+$ and $\R_-$, on which $d=d_X$, and to  
\eqref{eq:13}. 

By uniqueness of the geodesic ray starting from the root, these path
must be one and only, so $\gamma_+(r)=\gamma_-(r)$ for every $r\geq 0$. Let
$s<0<t$, and let $h_t=\inf_{[t,\infty)}X$ and
$h_s=\inf_{(-\infty,s]}X$. We define
$$\Gamma_t(r)=\inf\{u\geq t:X_u=X_t-r\}\, ,\qquad
\Gamma_s(r)=\sup\{u\leq s:X_u=X_s-r\}\, ,$$ which take finite values
respectively if $0\leq r\leq X_t-h_t$ and $0\leq r\leq
X_s-h_s$. The images $p(\Gamma_t)$ and
$p(\Gamma_s)$ are geodesic paths in $\TT$ respectively from
$p(t),p(s)$ to the points $\gamma_+(h_t)$ and $\gamma_-(h_s)$. This is
due to the fact that $\Gamma_t,\Gamma_s$ take their values
respectively in $\R_+$ and $\R_-$, on which the restrictions of $d$
and $d_X$ coincide, and to the fact that, 
$$d_X(\Gamma_t(r),\Gamma_t(r'))=|r-r'|\, ,\qquad r,r'\in [0,X_t-h_t]\, ,$$
as is easily checked, together with the similar identity for
$\Gamma_s$. By connecting $\gamma_+(h_t)$ and $\gamma_-(h_s)$ along
$\gamma_+=\gamma_-$, we can construct a path from $p(s)$ to $p(t)$
with length
$$(X_s-h_s)+(X_t-h_t)+|h_t-h_s|=X_s+X_t-2 h_t\wedge h_s=d_X(s,t)\, .$$
Therefore, we have obtained that $d(s,t)\leq d_X(s,t)$ also for
$st<0$. So we can again apply a re-rooting argument in this situation,
and conclude that $d=d_X$ everywhere, almost surely. 
\end{proof}

The proof of Theorem \ref{sec:infin-brown-loop-3} does not follow
directly from Lemma \ref{sec:infin-brown-loop-9}, due to the fact that
there could be, in principle, points of the infinite Brownian loop
that are visited at large times, but are close to the origin, a
phenomenon that is not detected by the compact-open topology used so
far. Therefore, the discussion from this point will be longer than in
Section \ref{sec:convergence}. In the sequel, we again assume without
loss of generality that the convergence in Lemma
\ref{sec:infin-brown-loop-9} holds almost surely. 

For every $A>0$, we denote by $\TT_X(A)=p_X([-A,A])$, which is a
compact subset of $\TT_X$ (in fact, it is an $\R$-tree in its own
right). 
Then the proof of \eqref{eq:4} given in Section \ref{sec:convergence}
generalizes immediately to the following: almost surely, for every $A>0$, 
$$\left(\rho_{(a)},a\, \RR_\infty(A/a^2)\right)  \underset{a\to
  0}{\longrightarrow}(X,\TT_X(A))\, .
$$
In turn, because of the fact that $\TT_X(A)$ is a length space
(it is indeed an $\R$-tree), this implies that the balls of radius $r$
in these spaces converge in the pointed Gromov-Hausdorff topology, as
a simple variation of Exercise 8.1.3 in \cite{burago_course_2001}: 
\begin{equation}
  \label{eq:7}
  B(a\, \RR_\infty(A/a^2),r)  \underset{a\to
  0}{\longrightarrow} B(\TT_X(A),r)\, ,
\end{equation}
with $B(M,r)=\{y\in M:d(x,y)\leq r\}$ denoting the closed ball centered at $x$ with radius $r$
in the pointed metric space $(M,d,x)$, and $B(M,r)$ is seen as a metric space pointed
at $o$ and endowed with the restriction of $d$. 

Note that from the transience of the Bessel processes of
dimension greater than $2$, for any $r,\eps>0$, the set $\TT_X(A)$
contains the ball $B_{d_X}(o_X,r)$ with probability at least $1-\eps$
if we choose $A$ large enough, namely if
$$\P\left(\inf_{|u|>A} X_u\leq r\right)\leq \eps\, .$$
On this likely event, one has $B(\TT_X(A),r)=B(\TT_X,r)$. We will be able to
conclude the proof of Theorem \ref{sec:infin-brown-loop-3} if we can
show that for any $r>0$, the set $\RR_\infty(A/a^2)$ contains the ball
$B(a\, \RR_\infty,r)$ with high probability, uniformly in $a$ small
enough:
\begin{equation}
\label{eq:8}
\lim_{A\to\infty}\limsup_{a\to 0}\P\left(B(a\, \RR_\infty,r)\not\subset \RR_\infty(A/a^2)\right)=0\, .
\end{equation}
Indeed, in this case, this shows that $B(a\,
\RR_\infty(A/a^2),r)=B(a\, \RR_\infty,r)$ with high probability
uniformly in $a$ small enough, so that \eqref{eq:7} implies that for
every $r>0$, 
$$ B(a\, \RR_\infty,r)  \underset{a\to
  0}{\longrightarrow} B(\TT_X,r)\, ,$$ in probability in the pointed
Gromov-Hausdorff topology, and this implies that $a\, \RR_\infty$
converges in probability to $\TT_X$ in the local Gromov-Hausdorff topology, as
wanted.

% \begin{lmm}
%   \label{sec:infin-brown-loop-7} For every $r,\eps>0$ there exists
%   $A>0$ so that 
% $$\limsup_{a\to 0}\P\left(\inf_{\R\setminus [-A/a^2,A/a^2]} a\, \rho_\infty \leq r\right)\leq
% \eps\, .$$
% \end{lmm}

% \begin{proof}        
It remains to prove \eqref{eq:8}. For this, we will use Proposition 5.3 in
\cite{anker_infinite_2002} and its proof, where it is shown that the
process $(d_H(o,b_\infty(t))^2,t\geq 0)$ is a diffusion $(Y_t,t\geq
0)$ in $\R_+$ satisfying the stochastic differential equation
\begin{equation}
  \label{eq:9}
  Y_t=Y_0+2\int_0^t\sqrt{Y_s}\, \d \beta_s+t+2\int_0^t g(Y_s)\, \d s\,
  ,
\end{equation}
where $\beta$ is a standard Brownian motion and $g$ is a nonnegative continuous
function that converges to $1$ at infinity by Proposition 8.2 in
\cite{anker_infinite_2002} (one has $g(x)=\chi(\sqrt{x})$ with the
notation therein).  
Recall that the (strong) solution of the stochastic differential
equation 
$$\ZZ^{(n)}_t=\ZZ^{(n)}_0+2\int_0^t\sqrt{\ZZ^{(n)}_s}\, \d \beta_s+n t$$
is a squared Bessel process of dimension $n$ (for any real number
$n>0$).  From this, and using standard comparison principles
\cite[Theorem IX.3.7]{revuz_continuous_1991}, one concludes that we
can couple the process $Y$ with $\ZZ^{(1)}$ (the square of a reflected
Brownian motion) in such a way that $Y\geq \ZZ^{(1)}$ almost
surely. In particular, $Y$ is a.s.\ unbounded, so that for every
$M>0$, the time $\tau_M=\inf\{t\geq 0:Y_t=M\}$ is a.s.\ finite. Let
$C$ be large enough so that $g(x)>3/4$ for every $x\geq C$, and for a
fixed $\eps>0$, let $M>C$ be such that
$$\P\left(\inf_{t\geq 0}\ZZ^{(5/2)}_t\geq C\, \Big|\, \ZZ^{(5/2)}_0=M\right)\geq
1-\eps\, .$$ By applying the Markov property at time $\tau_M$, the
process $(Y_{\tau_M+t},t\geq 0)$ satisfies \eqref{eq:9} starting
from the value $M$, and by the choice of $C$ it can be coupled with
the squared Bessel process $\ZZ^{(5/2)}$ starting from $M$ in such a
way that $\ZZ^{(5/2)}_t\leq Y_{\tau_M+t}$ on the event that
$\inf_{t\geq 0}\ZZ^{(5/2)}_t\geq C$ (so that the drift coefficient in
\eqref{eq:9} remains bounded from below by $5/2$). Finally, let $T$ be
large enough so that $\P(\tau_M>T)\leq \eps$. Upon a further
application of the Markov property at time $T$, we have shown that
outside the event
$$\mathcal{A}= \{\tau_M>T\}\cup \left(\left\{\inf_{t\geq
      0}\ZZ^{(5/2)}_t< C\right\}\cap \{\tau_M\leq T\}\right)$$ of
probability at most $2\eps$, we can couple $b_\infty$ with a 
Bessel process of dimension $5/2$, say $Z$, in such a way that
$d_H(o,b_\infty(t+T))\geq Z_t$ for every $t\geq 0$. Without loss of generality, we
may assume that $Z$ starts from $0$, again by standard monotone
coupling. 

Finally, one has
\begin{align*}
\P(B(a\, \RR_\infty,r)\not\subset\RR_\infty(A/a^2))&=\P(\exists\,
t\notin[-A/a^2,A/a^2]:  a\, \rho_\infty(t)\leq r)\\
&\leq 2\, \P(\exists\,
t>A: a\, \rho_\infty(t/a^2)\leq r)\\
&\leq 2\P(\mathcal{A}) + 2\P(\exists\,  t>A:a\, Z_{t/a^2-T}\leq r)\\
&\leq 4\eps+2\P(\exists\, t>A:Z_{t-a^2T}\leq r)\, ,
\end{align*}
where we used symmetry at the first step, the coupling with $Z$ at the
penultimate step, and the scaling property of Bessel processes at the
last step. Choosing $a_0$ small enough so that
$a_0^2T\leq 1$ say, we can find $A>1$ large enough so that 
$\P(\exists\, t>A:Z_{t-a^2T}\leq r)\leq \eps$ for every $a\in (0,a_0)$
by the transience of $Z$. 
This concludes the proof of \eqref{eq:8},
and thus of Theorem \ref{sec:infin-brown-loop-2}.

\bibliographystyle{abbrv}
\bibliography{biblio.bib}

\begin{thebibliography}{10}

\bibitem{aidekon_scaling_2015}
E.~A{\"\i}d{\'e}kon and L.~de~Raph{\'e}lis.
\newblock Scaling limit of the recurrent biased random walk on a
  {Galton}-{Watson} tree.
\newblock {\em arXiv:1509.07383 [math]}, Sept. 2015.
\newblock arXiv: 1509.07383.

\bibitem{aldous_continuum_1991-1}
D.~Aldous.
\newblock The continuum random tree. {I}.
\newblock {\em Ann. Probab.}, 19(1):1--28, 1991.

\bibitem{aldous_continuum_1991}
D.~Aldous.
\newblock The continuum random tree. {II}. {An} overview.
\newblock In {\em Stochastic analysis ({Durham}, 1990)}, volume 167 of {\em
  London {Math}. {Soc}. {Lecture} {Note} {Ser}.}, pages 23--70. Cambridge Univ.
  Press, Cambridge, 1991.

\bibitem{aldous_continuum_1993}
D.~Aldous.
\newblock The continuum random tree. {III}.
\newblock {\em Ann. Probab.}, 21(1):248--289, 1993.

\bibitem{anker_infinite_2002}
J.-P. Anker, P.~Bougerol, and T.~Jeulin.
\newblock The infinite {Brownian} loop on a symmetric space.
\newblock {\em Rev. Mat. Iberoamericana}, 18(1):41--97, 2002.

\bibitem{anker-ostellari}
J.-P. Anker and P.~Ostellari.
\newblock The heat kernel on noncompact sysmmetric spaces.
\newblock In {\em Lie groups and symmetric spaces}. 2003.

\bibitem{billingsley_convergence_1999}
P.~Billingsley.
\newblock {\em Convergence of probability measures}.
\newblock Wiley {Series} in {Probability} and {Statistics}: {Probability} and
  {Statistics}. John Wiley \& Sons, Inc., New York, second edition, 1999.
\newblock A Wiley-Interscience Publication.

\bibitem{bougerol_brownian_1999}
P.~Bougerol and T.~Jeulin.
\newblock Brownian bridge on hyperbolic spaces and on homogeneous trees.
\newblock {\em Probab. Theory Related Fields}, 115(1):95--120, 1999.

\bibitem{bridson_metric_1999}
M.~R. Bridson and A.~Haefliger.
\newblock {\em Metric spaces of non-positive curvature}, volume 319 of {\em
  Grundlehren der {Mathematischen} {Wissenschaften} [{Fundamental} {Principles}
  of {Mathematical} {Sciences}]}.
\newblock Springer-Verlag, Berlin, 1999.

\bibitem{burago_course_2001}
D.~Burago, Y.~Burago, and S.~Ivanov.
\newblock {\em A course in metric geometry}, volume~33 of {\em Graduate
  {Studies} in {Mathematics}}.
\newblock American Mathematical Society, Providence, RI, 2001.

\bibitem{caraceni_scaling_2014}
A.~Caraceni.
\newblock The {Scaling} {Limit} of {Random} {Outerplanar} {Maps}.
\newblock {\em arXiv:1405.1971 [math]}, May 2014.
\newblock arXiv: 1405.1971.

\bibitem{curien_crt_2015}
N.~Curien, B.~Haas, and I.~Kortchemski.
\newblock The {CRT} is the scaling limit of random dissections.
\newblock {\em Random Structures Algorithms}, 47(2):304--327, 2015.

\bibitem{curien_brownian_2014}
N.~Curien and J.-F. Le~Gall.
\newblock The {Brownian} plane.
\newblock {\em J. Theoret. Probab.}, 27(4):1249--1291, 2014.

\bibitem{drmota_random_2009}
M.~Drmota.
\newblock {\em Random trees}.
\newblock SpringerWienNewYork, Vienna, 2009.
\newblock An interplay between combinatorics and probability.

\bibitem{duquesne_continuum_2005}
T.~Duquesne.
\newblock Continuum tree limit for the range of random walks on regular trees.
\newblock {\em Ann. Probab.}, 33(6):2212--2254, 2005.

\bibitem{gouezel_local_2014}
S.~Gou{\"e}zel.
\newblock Local limit theorem for symmetric random walks in {Gromov}-hyperbolic
  groups.
\newblock {\em J. Amer. Math. Soc.}, 27(3):893--928, 2014.

\bibitem{gouezel_random_2013}
S.~Gou{\"e}zel and S.~P. Lalley.
\newblock Random walks on co-compact {Fuchsian} groups.
\newblock {\em Ann. Sci. {\'E}c. Norm. Sup{\'e}r. (4)}, 46(1):129--173 (2013),
  2013.

\bibitem{grigoryan-noguchi}
A.~Grigor'yan and M.~Noguchi.
\newblock The heat kernel on hyperbolic space.
\newblock {\em Bulletin London Math. Soc.}, 30(6):643--650, 1998.

\bibitem{kortchemski_invariance_2012}
I.~Kortchemski.
\newblock Invariance principles for {Galton}-{Watson} trees conditioned on the
  number of leaves.
\newblock {\em Stochastic Process. Appl.}, 122(9):3126--3172, 2012.

\bibitem{le_gall_random_2005}
J.-F. Le~Gall.
\newblock Random trees and applications.
\newblock {\em Probab. Surv.}, 2:245--311, 2005.

\bibitem{le_gall_topological_2007}
J.-F. Le~Gall.
\newblock The topological structure of scaling limits of large planar maps.
\newblock {\em Invent. Math.}, 169(3):621--670, 2007.

\bibitem{le_gall_uniqueness_2013}
J.-F. Le~Gall.
\newblock Uniqueness and universality of the {Brownian} map.
\newblock {\em Ann. Probab.}, 41(4):2880--2960, 2013.

\bibitem{marckert_limit_2006}
J.-F. Marckert and A.~Mokkadem.
\newblock Limit of normalized quadrangulations: the {Brownian} map.
\newblock {\em Ann. Probab.}, 34(6):2144--2202, 2006.

\bibitem{panagiotou_scaling_2015}
K.~Panagiotou and B.~Stufler.
\newblock Scaling limits of random {P}{\textbackslash}'olya trees.
\newblock {\em arXiv:1502.07180 [math]}, Feb. 2015.
\newblock arXiv: 1502.07180.

\bibitem{panagiotou_scaling_2014}
K.~Panagiotou, B.~Stufler, and K.~Weller.
\newblock Scaling {Limits} of {Random} {Graphs} from {Subcritical} {Classes}.
\newblock {\em arXiv:1411.1865 [math]}, Nov. 2014.
\newblock arXiv: 1411.1865.

\bibitem{pitman_schros_2015}
J.~Pitman and D.~Rizzolo.
\newblock Schr{\"o}der's problems and scaling limits of random trees.
\newblock {\em Trans. Amer. Math. Soc.}, 367(10):6943--6969, 2015.

\bibitem{revuz_continuous_1991}
D.~Revuz and M.~Yor.
\newblock {\em Continuous martingales and {Brownian} motion}, volume 293 of
  {\em Grundlehren der {Mathematischen} {Wissenschaften} [{Fundamental}
  {Principles} of {Mathematical} {Sciences}]}.
\newblock Springer-Verlag, Berlin, 1991.

\bibitem{rizzolo_scaling_2015}
D.~Rizzolo.
\newblock Scaling limits of {Markov} branching trees and {Galton}-{Watson}
  trees conditioned on the number of vertices with out-degree in a given set.
\newblock {\em Ann. Inst. Henri Poincar{\'e} Probab. Stat.}, 51(2):512--532,
  2015.

\bibitem{stewart_range_2016}
A.~Stewart.
\newblock {\em On the scaling limit of the range of a random walk on regular
  tree}.
\newblock PhD thesis, University of Toronto, 2016.

\bibitem{stufler_scaling_2015}
B.~Stufler.
\newblock Scaling limits of random outerplanar maps with independent
  link-weights.
\newblock {\em arXiv:1505.07600 [math]}, May 2015.
\newblock arXiv: 1505.07600.

\end{thebibliography}

\end{document}